\documentclass[12pt]{article}
\usepackage{latexsym,amsmath,stackrel,color,amsthm,amssymb,epsfig,graphicx,mathrsfs}
\usepackage{tikz}
\usetikzlibrary{positioning}
\usepackage{graphicx}
\usepackage{amssymb}
\usepackage[left=1in,top=1in,right=1in,bottom=1in]{geometry}
\usepackage[linktocpage=true]{hyperref}
\usepackage{setspace}
\usepackage{amssymb, amsmath, amsthm, graphicx,mathrsfs}
\usepackage{caption,cite}
\usepackage{comment,caption}
\usepackage{soul}

\newtheorem{thm}{Theorem}

\newtheorem{lem}[thm]{Lemma}
\newtheorem{conjecture}{Conjecture}

\renewcommand{\l}{\left}
\renewcommand{\r}{\right}

\renewcommand{\t}{\text}

\newfont{\footsc}{cmcsc10 at 8truept}
\newfont{\footbf}{cmbx10 at 8truept}
\newfont{\bigrm}{cmr12 scaled\magstep4}
\newfont{\footrm}{cmr10 at 10truept}
\definecolor{DarkGreen}{rgb}{0,0.6,0}
\definecolor{DarkBlue}{rgb}{0,0,0.7}

\newcommand{\F}{\mathcal{F}}

\renewcommand{\P}{\mathbb{P}}

\begin{document}
\title{Ramsey Numbers \\ for \\ Non-trivial Berge Cycles}

\medskip

\author{
Jiaxi Nie\footnote{E-mail: jin019@ucsd.edu} \qquad \qquad Jacques Verstra\" ete\footnote{Research supported by NSF Award DMS-1952786. E-mail: jacques@ucsd.edu} \\
\\
Department of Mathematics \\
University of California, San Diego \\
9500 Gilman Drive \\
La Jolla CA 92093-0112.
}

\maketitle

\begin{abstract}
 In this paper, we consider an extension of cycle-complete graph Ramsey numbers to Berge cycles in hypergraphs: for $k \geq 2$, a {\em non-trivial Berge $k$-cycle} is a family of sets $e_1,e_2,\dots,e_k$ such that $e_1 \cap e_2, e_2 \cap e_3,\dots,e_k \cap e_1$ has a system of distinct
representatives and $e_1 \cap e_2 \cap \dots \cap e_k = \emptyset$. In the case that all the sets $e_i$ have size three, let $\mathcal{B}_k$ denotes the family of all non-trivial Berge $k$-cycles. The {\em Ramsey numbers} $R(t,\mathcal{B}_k)$ denote the minimum $n$ such that every $n$-vertex $3$-uniform hypergraph contains either a non-trivial Berge $k$-cycle or an independent set of size $t$. We prove
\[ R(t, \mathcal{B}_{2k}) \leq t^{1 + \frac{1}{2k-1} + \frac{4}{\sqrt{\log t}}}\]
and moreover, we show that if a conjecture of Erd\H{o}s and Simonovits~\cite{ES} on girth in graphs is true, then this is tight up to a factor $t^{o(1)}$ as $t \rightarrow \infty$.
\end{abstract}

\section{Introduction}

Let $\F$ be a family of $r$-graphs and $t\ge1$. The Ramsey numbers $R(t,\F)$ denote the minimum $n$ such that every $n$-vertex $r$-graph contains
either a hypergraph in $\F$ or an independent set of size $t$. For $k \geq 2$, a {\em Berge $k$-cycle} is a family of sets $e_1,e_2,\dots,e_k$ such that $e_1 \cap e_2, e_2 \cap e_3,\dots,e_k \cap e_1$ has a system of distinct representatives,  and a Berge cycle is {\em non-trivial} if $e_1 \cap e_2 \cap \dots \cap e_k = \emptyset$.
Let $\mathcal{B}_k^r$ denote the family of non-trivial Berge $k$-cycles all of whose sets have size $r$. When $r = 2$, $\mathcal{B}_k^2 = \{C_k\}$, where $C_k$ denotes the graph
cycle of length $k$. In this paper, we let $\mathcal{B}_k=\mathcal{B}_k^3$.

\medskip

It is a notoriously difficult problem to determine even the order of magnitude of $R(t,C_k)$ -- the cycle-complete graph Ramsey numbers.
Kim~\cite{K} proved $R(t,C_3) = \Omega(t^2/\log t)$, which gives the order of magnitude of $R(t,C_3)$ when combined with the results of Ajtai, Koml\'os and Szemer\'edi~\cite{AKS} and Shearer~\cite{She}. The current state-of-the-art results on $R(t,C_3)$ are due to Fiz Pontiveros, Griffiths and Morris~\cite{FGM} and
Bohman and Keevash~\cite{BK}, using the random triangle-free process, which determines $R(t,C_3)$ up to a small constant factor.
$$
(\frac{1}{4}-o(1))\frac{t^2}{\log t}\le R(t,C_3)\le (1+o(1))\frac{t^2}{\log t}.
$$
The case $R(t,C_4)$ is the subject of a notorious conjecture of Erd\H{o}s~\cite{Fan}, where he conjectured that $R(t,C_4)=o(t^{2-\epsilon})$ for some $\epsilon>0$. The current best upper bounds on $R(t,C_{2k})$ is 
$$O\l(\l(\frac{t}{\log t}\r)^{k/(k-1)}\r),$$ 
which come from the work of Caro, Li, Rousseau and Zhang~\cite{CL}. For $R(t,C_{2k+1})$, the best upper bound is 
$$O\l(\frac{t^{(k+1)/k}}{\log^{1/k}t}\r)$$ 
due to Sudakov~\cite{S}. Recent results using pseudorandom graphs by Mubayi and the second author~\cite{MV} give the best lower bounds on cycle-complete graph Ramsey numbers:
$$
R(C_k,n)=\Omega\l(\frac{t^{(k-1)/(k-2)}}{\log^{2/(k-2)}t}\r).
$$
In particular, via random block constructions, they show that
$$
R(C_5,t)\ge (1+o(1))t^{{11}/{8}},\ \ R(C_7,t)\ge (1+o(1))t^{{11}/{9}}.
$$

For $k \geq 3$, a {\em loose $k$-cycle} is a non-trivial Berge $k$-cycle, denoted $C_k^r$, with sets $e_1,e_2,\dots,e_k$ of size $r$ such that $|e_1 \cap e_2| = 1$, $|e_2 \cap e_3| = 1, \dots, |e_k \cap e_1| = 1$, and for any other pairs of edges $e_i$,$e_j$, $e_i \cap e_j=\emptyset$.
Ramsey type problems for loose cycles in $r$-graphs have been studied extensively~\cite{BG,CGJ,DLS,FJ,GL1,GL2,K,KMV,KMV2,M,MV}. For $r$-uniform hypergraphs with $r\ge3$, Kostochka, Mubayi and the second author~\cite{KMV} proved for all $r \geq 3$, there exist constants $a,b > 0$ such that
\begin{equation}\label{berge-triangle}
    \frac{at^{\frac{3}{2}}}{(\log t)^{\frac{3}{4}}} \leq R(t,C_3^r) \leq bt^{\frac{3}{2}},
\end{equation}
The following conjecture was proposed in~\cite{KMV}:

\begin{conjecture}\label{mainconj}
For $r,k \geq 3$,
\begin{equation}
R(t,C_k^r)=t^{\frac{k}{k-1} + o(1)}.
\end{equation}
\end{conjecture}

The conjecture is true for $k = 3$ due to (\ref{berge-triangle}). It is shown in~\cite{NV} that $R(t,C_4^3) \leq t^{4/3 + o(1)}$. M\'{e}roueh~\cite{M} showed $R(t,C_k^3) = O(t^{1 + 1/\lfloor (k + 1)/2\rfloor})$ for $k \geq 3$ and $R(t,C_k^r) = O(t^{1 + 1/\lfloor k/2\rfloor})$ for $r\ge 4$ and every odd integers $k\ge 5$, improving earlier results of
Collier-Cartaino, Graber and Jiang~\cite{CGJ}. Conjecture \ref{mainconj} motivates our current study of non-trivial Berge $k$-cycles.
In support of the above conjecture, we prove the following result for non-trivial Berge cycles of even length:
\begin{thm}\label{thm:nbc}
For $k \geq 3$, and $t$ large enough,
\begin{equation*}
R(t,\mathcal{B}_{2k}) \leq t^{\frac{2k}{2k-1} + \frac{4}{\sqrt{\log t}}}.
\end{equation*}
\end{thm}

Erd\H{o}s and Simonovits~\cite{ES} conjectured that there exists an $n$-vertex graph of girth more than $2k$ with $\Theta(n^{1+1/k})$ edges. This notoriously difficult conjecture remains open, except when $k\in\{2,3,5\}$, largely due to the existence of generalized polygons~\cite{Benson,vanMaldeghem,V}. Towards this conjecture, Lazebnik, Ustimenko and Woldar~\cite{LUW} gave the densest known construction, which has $\Omega(n^{1+2/(3k-2)})$ edges.
We prove the following theorem relating this conjecture to lower bounds on Ramsey numbers for non-trivial Berge cycles:

\begin{thm}\label{thm:rbc}
Let $k\ge2$, $r\ge3$. Suppose there exists an $n$-vertex graph of girth more than $2k$ with $cn^{1+1/k}$ edges for any integer $n$ large enough and some positive constant $c$. Then for $t$ large enough and some positive constant $c_{k,r}$ dependent on $k$ and $r$,
\begin{equation}
R(t,\mathcal{B}^r_{k}) \geq c_{k,r}\l(\frac{t}{\log t}\r)^{\frac{k}{k-1}}.
\end{equation}
\end{thm}

This shows that if the Erd\H{o}s-Simonovits Conjecture is true, then Theorem \ref{thm:nbc} is tight up to a $t^{o(1)}$ factor. Indeed, following the proof of Theorem~\ref{thm:rbc},
the known construction of Lazebnik, Ustimenko and Woldar~\cite{LUW} would give a weaker lower bound of $\Omega((t/\log t)^{(3k-2)/(3k-4)})$.

\medskip

Let $B_k$ be the family of $3$-uniform Berge $k$-cycles without non-triviality. Random graphs together with the Lov\'asz local lemma give $R(t, B_k)\ge t^{(2k-2)/(2k-3)-o(1)}$, see~\cite{AS} for similar computation. We prove the following theorem, which gives a substantially better lower bound for $B_4$ if the Erd\H{o}s-Simonovits Conjecture is true.

\begin{thm}\label{thm:4bc}
Suppose there exists an $n$-vertex graph of girth more than $8$ with $c_1n^{5/4}$ edges for any integer $n$ large enough and some positive constant $c_1$. Then for $t$ large enough and some positive constant $c_2$,
$$
R(t,B_4)\ge \l(\frac{c_2t}{\sqrt{\log t}}\r)^{16/13}.
$$
\end{thm}

In fact, this is also a lower bound for $R(t,\{B_2,B_3,B_4\})$. A natural $3$-uniform analog of the Erd\H{o}s-Simovits conjecture is that there exist $n$-vertex $\{B_2,B_3,\dots,B_{k}\}$-free 3-graphs with $n^{1+1/\lfloor k/2\rfloor-o(1)}$ edges. This is true for $k=3$ due to Ruzsa and Szemeredi~\cite{RS}. The proof of Theorem~\ref{thm:4bc} makes use of the fact that there exist $n$-vertex $\{B_2,B_3,B_4\}$-free 3-graphs with $\Omega(n^{3/2})$ edges, that is, the conjecture is true for $k=4$, which is due to Lazebnik and the second author~\cite{LV}. More generally, following the proof of Theorem~\ref{thm:4bc}, if the $3$-uniform analog of the Erd\H{o}s-Simonovits Conjecture is true, then we have $R(t,\{B_2,B_3,\dots,B_{2k}\})\ge t^{2k^2/(2k^2-k-2)-o(1)}$ and $R(t,\{B_2,B_3,\dots,B_{2k+1}\})\ge t^{2k(k-1)/(2k^2-3k-1)-o(1)}$, which are substantially better than the lower bounds obtained by random graphs.    

We prove Theorem~\ref{thm:nbc} in Section~\ref{section:nbc}, Theorem~\ref{thm:rbc} in Section~\ref{section:rbc} and Theorem~\ref{thm:4bc} in Section~\ref{section:4bc}. Theorem \ref{thm:rbc} is valid for
all values of $k\ge 2$ and $r\ge 3$, while Theorem \ref{thm:nbc} only works for even values of $k$ and $r=3$. We believe that Theorem \ref{thm:nbc} should extend to odd values of $k$ and all $r \geq 3$:

\begin{conjecture}
For all $r,k \geq 3$,
\begin{equation}
R(t,\mathcal{B}_k^r) \leq t^{\frac{k}{k-1} + o(1)}.
\end{equation}
\end{conjecture}

\medskip

{\bf Notation and terminology.} For a hypergraph $H$, let $V(H)$ denote the vertex set of $H$, $v(H) = |V(H)|$ and let $|H|$ be the number of edges in $H$. If all edges of $H$ have size $r$, we say $H$ is an {\em $r$-uniform hypergraph}, or an {\em $r$-graph} for short. For $v \in V(H)$, let $d_H(v) = |\{e \in H : v \in e\}|$ be the {\em degree of $v$} in $H$. We denote the average degree of $H$ by $d(H)$, denote the minimum degree of $H$ by $\delta(H)$, and the maximum degree of $H$ by $\Delta(H)$. For $u,v \in V(H)$, let $d_H(u,v) = |\{w : uvw \in H\}|$ denote the {\em codegree} of the pair $\{u,v\}$. An {\em independent set} in a hypergraph is a set of vertices containing no edge of the hypergraph. Let $\alpha(H)$ denote the largest size of an independent set in a hypergraph $H$.

\section{Proof of Theorem~\ref{thm:rbc}}\label{section:rbc}

We will use the following lemma to get a large bipartite subgraph with large minimum degree and small maximum degree:

\begin{lem}\label{deg}
Let $k \geq 3$, $c>0$, and let $G$ be an $n$-vertex graph of girth more than $2k$ with more than $2cn^{1+1/k}$ edges. Then there exists a bipartite subgraph $G'$ of $G$ such that $\delta(G')\ge cn^{1/k}$, $\Delta(G')\le n^{1/k}/c^{k-1}$, and $v(G')\ge c^kn$.
\end{lem}

\begin{proof}
A maximum cut of $G$ gives a bipartite subgraph with at least $cn^{1+1/k}$ edges. A subgraph $G'$ of this bipartite subgraph of minimum degree at least $cn^{1/k} + 1$ may be obtained by repeatedly removing vertices of degree at most $cn^{1/k}$. Let $\Delta:= \Delta(G')$ be the maximum degree of $G'$, and let $v$ be a vertex of
maximum degree, then the number of vertices at distance $k$ from $v$ is at least $\Delta c^{k-1}n^{(k - 1)/k}$,
since $G$ has girth larger than $2k$. In particular, $\Delta c^{k-1}n^{(k - 1)/k}\le n$ and so $\Delta \le n^{1/k}/c^{k-1}$.
The number of vertices in $G'$ is at least $c^k n$, since $G'$ has minimum degree at least $cn^{1/k} + 1$ and girth larger than $2k$.
\end{proof}

Let $r\ge2$, a {\em star} with vertex set $V$ is an $r$-graph on $V$ consisting of all edges containing a fixed vertex of $V$, i.e., the edge set of a star is $\{e\subset V: |e|=r, v\in e\}$ for some vertex $v\in V$. Let integers $d\ge m$ and let $S_{d,m}$ be a $d$-vertex $r$-graph consisting of $m$ vertex-disjoint stars of size $\lfloor d/m\rfloor$ or $\lceil d/m\rceil$.
\begin{lem}\label{prob}
Let integer $r\ge2$, and let integers $d\ge m$.The probability that a uniformly chosen set of $s$ vertices of $S_{d,m}$ is independent is at most 
$$
\exp\l(-\frac{m(s-rm)}{2d}\r).
$$
\end{lem}

\begin{proof}
Let the vertex sets of these stars be $V_1,V_2,\dots,V_m$. The probability that a uniformly chosen set of $s_i$ vertices in $V_i$ is independent in $S_{d,m}$ is at most $1-s_i/\lceil d/m\rceil\le 1-ms_i/2d$ if $s_i\ge r$, and is $1$ if $s_i<r$. Hence, this probability is at most $1-m(s_i-r)/2d$ for $0\le s_i\le d$. Therefore a uniformly chosen set $I\subset S_{d,m}$ of $s$ vertices with $|I\cap V_i|=s_i$ is independent with probability at most
$$
\prod_{i=1}^m\l(1-\frac{m(s_i-r)}{2d}\r)\le \exp\l(-\sum_{i=1}^m\frac{m(s_i-r)}{2d}\r)=\exp\l(-\frac{m(s-rm)}{2d}\r).
$$
\end{proof}

Now we are ready to prove Theorem~\ref{thm:rbc}.

\begin{proof}[Proof of Theorem~\ref{thm:rbc}]
It suffices to show that for $n$ large enough, there exists an $n$-vertex $\mathcal{B}^r_k$-free $r$-graph with independence number $O(n^{1-\frac{1}{k}}\log n)$. Let $G$ be an $n$-vertex graph of girth more than $2k$ with $2cn^{1+1/k}$ edges for some positive constant $c$. By Lemma~\ref{deg}, there exists a bipartite subgraph $G'$ of $G$ with at least $N=c^k n$ vertices, minimum degree at least $cn^{1/k}$ and maximum degree at most $n^{1/k}/c^{k-1}$. Let $X,Y$ be the parts of this bipartite graph where $|Y|\ge|X|$. Let $m=8\log n/c^{k}$. We form an $r$-graph $H$ with vertex set $Y$ by placing a random copy of $S_{d(x),m}$ on the vertex set $N_{G'}(x)$, the neighborhood of $x$ in $G'$, independently for each $x\in X$. Since $G'$ has girth more than $2k$, it is straightforward to check that $H$ does not contain any non-trivial Berge $k$-cycle. We now compute the expected number of independent sets of size $t=rmn^{1-1/k}/c^{k+1}$ in $H$. Clearly, $\log t\ge (1-1/k)\log n$. If $H$ has no independent set of size $t$ with positive probability, then since $v(H)\ge N/2$, we find that
\[ R(t,\mathcal{B}^r_k)\ge N/2 \ge \frac{c^k}{2}\l(\frac{c^{2k+1}t}{8r\log n}\r)^{\frac{k}{k-1}}\ge c_{k,r}\l(\frac{t}{\log t}\r)^{\frac{k}{k-1}},\]
for some positive constant $c_{k,r}$. This is enough to prove Theorem \ref{thm:rbc}.

\medskip

For an independent $t$-set $I$ in $H$, $I\cap N_{G'}(x)$ is an independent set in $S_{d(x),m}$ for all $x\in X$. Since these events are independent, setting $s(x)=|I\cap N_{G'}(x)|$, and applying Lemma~\ref{prob} gives:
\begin{equation*}
    \P(\t{$I$ independent in $H$})\le\prod_{x\in X}\exp\l(-\frac{m(s(x)-rm)}{2d(x)}\r)=\exp\l(-\sum_{x\in X}\frac{ms(x)}{2d(x)}+\sum_{x\in X}\frac{rm^2}{2d(x)}\r).
\end{equation*}
For every $x\in X$, $cn^{1/k}\le d(x)\le n^{1/k}/c^{k-1}$ and therefore
\begin{equation*}
    \P(\t{$I$ independent in $H$})\le\exp\l(-\frac{c^{k-1}m\sum_{x\in X}s(x)}{2n^{1/k}}+\frac{|X|rm^2}{2cn^{1/k}}\r).
\end{equation*}
Now $\sum_{x\in X}s(x)$ is precisely the number of edges of $G'$ between $X$ and $I$. Since every vertex in $I$ has degree at least $cn^{1/k}$, this number of edges is at least $cn^{1/k}t=rmn/c^{k}$. Consequently, using $|X|<n/2$,
\begin{equation*}
    \P(\t{$I$ independent in $H$})\le\exp\l(-\frac{c^{k}mt}{2}+\frac{c^{k}mt}{4}\r)=\exp\l(-\frac{c^{k}mt}{4}\r).
\end{equation*}
The expected number of independent sets of size $t$ is at most
\begin{equation*}
    \binom{n}{t}\exp\l(-\frac{c^{k}mt}{4}\r)<\exp\l(t\log n-\frac{c^{k}mt}{4}\r)=\exp\l(-t\log n\r).
\end{equation*}
This is vanishing as $n\rightarrow\infty$, and the proof of Theorem~\ref{thm:rbc} is complete.
\end{proof}

\section{Proof of Theorem~\ref{thm:4bc}}\label{section:4bc}

Lazebnik and the second author~\cite{LV} showed that there exist $n$-vertex $B_4$-free 3-graphs with $(1/6+o(1))n^{3/2}$ triples. More specifically, for $n$ large enough, there exists a linear $n$-vertex $B_4$-free $3$-graphs $J_n$ with $n^{3/2}/10$ triples and maximum degree at most $n^{1/2}$. We want to find an upper bound for the probability that a random $s$-set is independent in $J_n$. We make use of the following lemma, where we make no effort to optimize the constants.

\begin{lem}\label{lem:4bc random set}
Let $n$, $s$ be integers such that $s<\sqrt{n}/2$. For $n$ large enough, the probability that a uniformly chosen set of $s$ vertices of $J_n$ is independent is at most
$$
\exp\l(-\frac{s^3-216}{80n^{3/2}}\r).
$$
When $s\ge \sqrt{n}/2$, the probability is at most $639/640$.
\end{lem}

\begin{proof}
This is trivial when $s<6$. When $6<s<\sqrt{n}/2$, let $X$ be the uniformly chosen $s$-set. For any edge $e\in E(J_n)$, let $A_e$ be the event that $e\in X$. Then by inclusion-exclusion principle, for $n$ large enough, the probability that $X$ is not independent is at least
$$
\begin{aligned}
&\sum_{e\in E(J_n)}\P(A_e)-\sum_{\{e,f\}\subset E(J_n)}\P(A_e\wedge A_f)\\
\ge&\frac{1}{\binom{n}{s}}\l(\frac{n^{3/2}}{10}\binom{n-3}{s-3}-n\binom{n^{1/2}}{2}\binom{n-5}{s-5}- \binom{n^{3/2}/10}{2}\binom{n-6}{s-6}\r)\\
\ge&\frac{s^3}{40n^{3/2}}\l(1-\frac{4s^3}{n^{3/2}}\r)\\
\ge&\frac{s^3}{80n^{3/2}}.
\end{aligned}
$$
Therefore, for $s>6$ and $n$ large enough, the probability that $X$ is independent is at most
$$
1-\frac{s^3}{80n^{3/2}}\le \exp\l(-\frac{s^3}{80n^{3/2}}\r)<\exp\l(-\frac{s^3-216}{80n^{3/2}}\r).
$$
When $s\ge \sqrt{n}/2$, the probability is at most
$$
1-\frac{(\sqrt{n}/2)^3}{80n^{3/2}}=\frac{639}{640}.
$$
\end{proof}

Now we are ready to prove Theorem~\ref{thm:4bc}.

\begin{proof}[Proof of Theorem~\ref{thm:4bc}]
Let $G$ be an $n$-vertex graph of girth more than $8$ with $2c_1n^{5/4}$ edges for some positive constant $c_1$. By Lemma $4$, there exists a bipartite subgraph $G'$ of $G$ with at least $N=c_1^4n$ vertices, minimum degree at least $c_1n^{1/4}$ and maximum degree at most $n^{1/4}/c_1^3$. Let $X$, $Y$ be the parts of this bipartite graph where $|Y|\ge|X|$. We form a $3$-graph $H$ with vertex set $Y$ by placing a random copy of $J_{d(x)}$ on the vertex set $N_{G'}(x)$, the neighborhood of $x$ in $G$, independently for each $x\in X$. Since $G$ has girth more than $2k$, it is straightforward to check that $H$ does not contain any Berge $4$-cycle. Let $m=8c_1^{1/4}\sqrt{\log n}$, and let $t=mn^{13/16}$. Clearly, $\log t>13\log n/16$. If $H$ has no independent sets of size $t$ with positive probability, then since $v(H)\ge N/2$, we conclude that
$$
R(t,B_4)\ge N/2\ge \frac{c_1^4}{2}\l(\frac{t}{8c_1^{1/4}\sqrt{\log n}}\r)^{16/13}\ge c_2\l(\frac{t}{\sqrt{\log t}}\r)^{16/13},
$$
for some positive constant $c_2$. This is enough to prove Theorem~\ref{thm:4bc}.

Let $A$ be a $t$-set in $Y$, and let $X_A=\{x\in X||N_{G'}(x)\cap A|\ge\sqrt{t}/2\}$, $\overline{X}_A=X\backslash A$. We now evaluate the probability that $A$ is independent in $H$ in two cases.\\
\textbf{Case 1:} When $|X_A|<n^{5/6}$. Since the induced bipartite subgraph of $G'$ on $X_A\cup A$ has girth $8$, the number of edges of $G'$ between $X_A$ and $A$ is less than $(n^{5/6})^{5/4}=n^{25/24}$. If $A$ is independent in $H$, then $N_{G'}(x)\cap A$ is also independent in $J_{d(x)}$ for all $x\in X$. Since these events are independent, setting $s(x)=|N_{G'}(x)\cap A|$, and applying Lemma~\ref{lem:4bc random set} gives
$$
\begin{aligned}
\P(A\text{ independent in }H)&\le \prod_{x\in\overline{X}_A}\exp\l(-\frac{s(x)^3-216}{80d(x)^{3/2}}\r)\\
&=\exp\l(-\sum_{x\in\overline{X}_A}\frac{s(x)^3}{80d(x)^{3/2}}+\sum_{x\in\overline{X}_A}\frac{27}{10d(x)^{3/2}}\r).
\end{aligned}
$$
For every $x\in X$, $c_1n^{1/4}\le d(x)\le n^{1/4}/c_1^{3}$ and hence together with Jenson's inequality we have
$$
\begin{aligned}
\P(A\text{ independent in }H)&\le \exp\l(-\frac{c_1^{9/2}\sum_{x\in\overline{X}_A}s(x)^3}{80n^{3/8}}+\frac{27|\overline{X}_A|}{10c_1^{3/2}n^{3/8}}\r)\\
&\le\exp\l(-\frac{c_1^{9/2}(\sum_{x\in\overline{X}_A}s(x))^3}{80n^{3/8}|\overline{X}_A|^2}+\frac{27|\overline{X}_A|}{10c_1^{3/2}n^{3/8}}\r).
\end{aligned}
$$
Note that $\sum_{x\in\overline{X}_A}s(x)$ is exactly the number of edges of $G'$ between $\overline{X}_A$ and $A$, which is at least $tc_1n^{1/4}-n^{25/24}=(1-o(1))c_1mn^{17/16}$. Also note that $|\overline{X}_A|<N/2=c_1^4n/2$. Consequently,
$$
\begin{aligned}
\P(A\text{ independent in }H)&\le \exp\l(-\frac{(1-o(1))m^3n^{13/16}}{20c_1^{1/2}}+\frac{27c_1^{5/2}n^{5/8}}{20}\r)\\
&<\exp\l(-\frac{m^3n^{13/16}}{32c_1^{1/2}}\r).
\end{aligned}
$$
\textbf{Case 2:} When $|X_A|\ge n^{5/6}$. Applying Lemma 6 gives
$$
\P(A\text{ independent in }H)\le (639/640)^{|X_A|}\le \exp(-n^{5/6}/640)< \exp\l(-\frac{m^3n^{13/16}}{32c_1^{1/2}}\r).
$$
In both cases we have $\P(A\text{ independent in }H)< \exp\l(-\frac{m^3n^{13/16}}{32c_1^{1/2}}\r)$. Therefore the expected number of independent sets of size $t$ in $H$ is at most
$$
\binom{n}{t}\exp\l(-\frac{m^3n^{13/16}}{32c_1^{1/2}}\r)<\exp\l(mn^{13/16}\log n-\frac{m^3n^{13/16}}{32c_1^{1/2}}\r)=\exp\l(-mn^{13/16}\log n\r).
$$
This is vanishing as $n\rightarrow\infty$, which completes the proof of Theorem~\ref{thm:4bc}.

\end{proof}

\section{Degrees, codegrees and independent sets}

We make use of the following elementary lemma, \
whose proof is a standard probabilistic argument, included for completeness:

\begin{lem}\label{alphabound}
 Let $d \geq 1$, and let $H$ be a 3-graph of average degree at most $d$. Then
\[
 \alpha(H) \geq \frac{2v(H)}{3d^{\frac{1}{2}}}.
\]
 \end{lem}

\begin{proof} Let $X$ be a subset of $V(H)$ whose elements are chosen independently with probability $p = d^{-1/2}$.
We can get an independent set by deleting a vertex for each edge of $H$ contained in $X$. Then the expected size of such independent set is at least
 \[ pv(H) - p^3|H| = pv(H) - \frac{p^3 dv(H)}{3} = \frac{2v(H)}{3d^{\frac{1}{2}}}.\]
 Hence, there must exist an independent set of size at least the desired lower bound, which completes the proof.
 \end{proof}

\begin{lem}\label{dmaxlemma}
Let $H$ be a $3$-graph on $n$ vertices, and $0 < \epsilon < 1/2$. Then there exists an induced subgraph $G$ of $H$ satisfying the following properties:
\begin{enumerate}
    \item $v(G)\ge n^{1-\frac{2}{\log_2(\frac{1}{\epsilon})}}$,
    \item $\Delta(G)\le \frac{d(G)}{\epsilon}$.
\end{enumerate}
\end{lem}

\begin{proof}
Let $H=G^{(0)}$. We do the following for $i\ge 0$. If $\Delta(G^{(i)})\le d(G^{(i)})/\epsilon$, we let $G=G^{(i)}$. Otherwise, iteratively delete vertices of $G^{(i)}$ with degree at least $d(G^{(i)})$. Each deleted vertex will result in the loss of at least $d(G^{(i)})$ edges. So we can delete at most
\begin{equation*}
    \frac{|G^{(i)}|}{d(G^{(i)})}=\frac{v(G^{(i)})\cdot d(G^{(i)})}{3\cdot d(G^{(i)})}=\frac{v(G^{(i)})}{3}<\frac{v(G^{(i)})}{2}
\end{equation*}

 vertices in this step. Let $G^{(i+1)}$ be the subgraph induced by the remaining vertices. Then we have $v(G^{(i+1)})>v(G^{(i)})/2$. If $\Delta(G^{(i+1)})\le d(G^{(i+1)})/\epsilon$, then we let $G=G^{(i+1)}$. Otherwise, we have
 \begin{equation*}
     d(G^{(i+1)})\leq \epsilon{\Delta(G^{(i+1)})}<\epsilon {d(G^{(i)})}.
 \end{equation*}
 Let $K=2\log_{1/\epsilon}n$. We must obtain an induced subgraph $G$ with $\Delta(G)\le d(G)/\epsilon$ after at most $K$ repetitions. Otherwise, after $K$ repetitions, since the average degree decreases by at least a factor of $\epsilon$ after each repetition, the remaining graph $G^{(K)}$ will have no edge, which satisfies the condition $\Delta(G^{(K)})\le d(G^{(K)})/\epsilon$. Suppose after $m\le K$ repetitions we have the desired induced subgraph $G$ with $\Delta(G)<d(G)/\epsilon$. Since the number of vertices decreases by at most a factor of $2$, we also have
 \begin{equation*}
     v(G)>\frac{n}{2^m}\ge n^{1-\frac{2}{\log_2(\frac{1}{\epsilon})}}.
 \end{equation*}
This completes the proof.
\end{proof}

We use the following slightly weaker version of a lemma due to M\'{e}roueh~\cite{M}; the lemma is in fact valid for $3$-graphs $H$ with no
loose $k$-cycles:

\begin{lem}\label{heavy}
Let $H$ be a $\mathcal{B}_k$-free $3$-graph. Then there exists a subgraph $H^*$ of $H$ such that $|H^*|>|H|/(3k^2)$ and each edge of $H^*$ contains a pair of codegree 1.
\end{lem}

\begin{proof} Given a $3$-graph $G$ and a pair of vertices ${x,y}$, we say that $\{x,y\}$ is {\em $G$-light} if $d_G(x,y)< k$. Let $G_1 = H$, and let $H_1$ consist of all edges of $G_1$ containing a $G_1$-light pair, and let $G_2=G_1\backslash H_1$. For $i \geq 2$, let $H_i$ consist
of all edges of $G_{i}$ containing a $G_{i}$-light pair, and let $G_{i+1} = G_{i} \backslash H_i$. Suppose for contradiction that
$G_{k}$ is not empty. Let $e_1=\{v_1,v_2,v_3\}$ be an edge in $G_k$, then by definition, $\{v_2,v_3\}$ is not a $G_{k-1}$-light pair, and hence, there exists an edge $e_2=\{v_2,v_3,v_4\}$ such that $v_4\not=v_1$. For $2\le i\le k-1$, let $e_i=\{v_i,v_{i+1},v_{i+2}\}$ be an edge in $G_{k+1-i}$. By definition, $\{v_{i+1},v_{i+2}\}$ is not a $G_{k-i}$-light pair, and hence, there exists an edge $e_{i+1}=\{v_{i+1},v_{i+2},v_{i+3}\}$ in $G_{k-i}$ such that $v_{i+3}$ is distinct from all $v_j$, $1\le j\le i$. Therefore, we have a tight path of length $k$ in $G_{1}=H$, that is, a hypergraph consisting of $k+2$ distinct vertices $v_i$, $1\le i\le k+2$, and $k$ edges $e_i=\{v_i,v_{i+1},v_{i+2}\}$, $1\le i\le k$. This is also a non-trivial Berge k-cycle. Indeed, when $k$ is even, $\{v_2, v_4, \dots, v_k, v_{k+1}, v_{k-1}, \dots, v_3\} $ forms a system of distinct representatives of $\{e_1\cap e_2, e_2\cap e_4, e_4\cap e_6,\dots, e_{k-2}\cap e_{k}, e_{k}\cap e_{k-1}, e_{k-1}\cap e_{k-3},\dots, e_{3}\cap e_1 \}$, and when $k$ is odd, $\{v_2, v_4, \dots, v_{k+1}, v_{k}, v_{k-2}, \dots, v_3\} $ forms a system of distinct representatives of $\{e_1\cap e_2, e_2\cap e_4, e_4\cap e_6,\dots, e_{k-3}\cap e_{k-1}, e_{k-1}\cap e_{k}, e_{k}\cap e_{k-2},\dots, e_{3}\cap e_1 \}$. This results in a contradiction, since $H$ is $\mathcal{B}_k$-free. Therefore, $G_k$ must be empty, and hence $H$ can be partitioned into $k-1$ subgraphs $H_i$, $1\le i\le k-1$, such that each $H_i$ consists of edges containing a $G_i$-light pair, which is also $H_i$-light. Let $H'$ be a subgraph $H_i$ with the most edges, then by the pigeonhole principle,
\begin{equation*}
    |H'|>\frac{|H|}{k}.
\end{equation*}
Now consider a graph $J$ whose vertex set is the set of $3$-edges of $H'$, and two $3$-edges of $H'$ form an edge of $J$ if they share an $H'$-light pair. It is easy to see that $J$ has maximum degree at most $3k-6$. Then we can greedily take an independent set of $J$ of size at least $v(J)/(3k-5)$, and this independent set correspond to a subgraph $H^*$ of $H'$ such that
\begin{equation*}
    |H^*|>\frac{|H'|}{3k-5}>\frac{|H|}{3k^2},
\end{equation*}
and each edge of $H^*$ contains a pair of codegree 1.
\end{proof}

\section{Proof of Theorem~\ref{thm:nbc}}\label{section:nbc}

A key ingredient of the proof of Theorem \ref{thm:nbc} is a supersaturation theorem for cycles in graphs: we make use of the following result proved by Simonovits~\cite{FS}
(see Morris and Saxton~\cite{MS} for stronger supersaturation):

\begin{lem}\label{FS}
For every $n,k\ge2$, there exist constants $\gamma,b_0>0$ such that for every $b\ge b_0$, any $n$-vertex graph $G$ with at  least $bn^{1+1/k}$ edges contains at least $\gamma b^{2k}n^2$ copies of $C_{2k}$.
\end{lem}

We next give a simple lemma which says that if a graph has many cycles of length $2k$ containing a fixed edge, then it has many edges.

\begin{lem}\label{bigcpn}
Let $G$ be a graph containing $m$ cycles of length $2k$, each containing an edge $e \in G$. Then
$|G| \geq m^{1/(k - 1)}/2$.
\end{lem}

\begin{proof}
For each cycle $C$ of length $2k$ containing $e$, let $M(C)$ be the perfect matching of $C$ containing $e$. Fixing a matching $M \subset G$ of size $k$ containing $e$, at most $(k - 1)!2^{k-1}$ cycles $C$ have $M(C) = M$. It follows that the number of distinct matchings $M \subset G$ of size $k$ containing $e$
is at least $m/(k - 1)!2^{k - 1}$, and therefore
\begin{equation*}
    {|G| - 1 \choose k - 1} \ge \frac{m}{(k  - 1)!2^{k-1}}.
\end{equation*}
We conclude $|G|^{k - 1} \geq m/2^{k - 1}$ and therefore $|G|\ge m^{1/(k-1)}/2$.
\end{proof}

Now we are ready to prove Theorem~\ref{thm:nbc}.
\begin{proof}[Proof of Theorem~\ref{thm:nbc}]
It suffices to show that for every large enough integer $n$, an $n$-vertex $\mathcal{B}_{2k}$-free 3-graph $H$ contains an independent set of size at least $n^{{(2k-1)}/{(2k)}-5/(2\sqrt{\log n})}$. By Lemma~\ref{dmaxlemma} with $\epsilon=\exp{(-\sqrt{\log_2 n})}$, we find an induced subgraph $H_0$ of $H$ with $n_0$ vertices, average degree $d_0$ and maximum degree $D_0$ such that $n_0\ge n^{1-2/\sqrt{\log_2 n}}$ and $D_0 < d_0/\epsilon$. By Lemma~\ref{heavy},
there is a subgraph $H_1$ of $H_0$ with at least $|H_0|/(4k^2)$ edges such that each edge of $H_1$ contains a pair of codegree 1 in $H_1$. Let $\chi : V(H_1) \rightarrow \{1,2,3\}$ be a random 3-coloring and let $H_2$ consist of all triples in $H_1$ such that the pair of vertices of colors 1 and 2 has codegree 1 in $H_1$ and the last
vertex in the triple has color 3. The probability that an edge in $H_1$ is also an edge in $H_2$ is at least $1/27$, and therefore the expected number of edges in $H_2$ is at least $|H_1|/27 \geq |H_0|/(108k^2)$.
Fix a coloring so that $|H_2| \geq |H_0|/(108k^2)$. Consider the bipartite graph $G$ comprising all pairs of vertices of colors 1 and 2 contained in an edge of $H_2$. Thus, $|G| = |H_2|$ and
$G$ has average degree $d_G \geq d_0/(108k^2)$. For convenience, let $b > 0$ be defined by $d_G = 2bn_0^{1/k}$ so $|G| = bn_0^{1 + 1/k}$. By Lemma~\ref{FS}, there exist constants $\gamma,b_0>0$ such that if $b>b_0$, then $G$ must contain at least $\gamma b^{2k}n_0^2$ copies of $C_{2k}$. Notice that we must have $1/\epsilon>b_0$ when $n$ is large enough. The proof is split into two cases.

{\bf Case 1.} $b \geq 1/\epsilon$. By the pigeonhole principle, there exists an edge $e$ such that the number of $C_{2k}$ containing $e$ in $G$ is at least
\[ \frac{2k\gamma b^{2k}n_0^2}{|G|} = 2k\gamma b^{2k-1}n_0^{1-\frac{1}{k}}.\]
Let $G'$ be the union of all $2k$-cycles in $G$ containing $e$. Then by Lemma~\ref{bigcpn}, for some constant $c$,
\begin{equation*}
 |G'| \geq  cb^{2+\frac{1}{k-1}}n_0^{\frac{1}{k}} = \frac{1}{2}{c b^{1+\frac{1}{k-1}}d_G} \geq \frac{1}{216k^2} c\epsilon^{-1-\frac{1}{k-1}}d_0> D_0
\end{equation*}
provided $n$ is large enough.
Let $C$ be a $2k$-cycle in $G$ containing $e$. Then there exist edges $e_1 \cup \{v_1\},e_2 \cup \{v_2\},\dots,e_{2k} \cup \{v_{2k}\}$
in $H_2$ where $e_1,e_2,\dots,e_{2k} \in C$ and $v_1,v_2,\dots,v_{2k}$ have color 3. Since $H_2$ is $\mathcal{B}_{2k}$-free, for some vertex $z$
we have $v_1 = v_2 = \dots = v_{2k} = z$. Since each cycle $C$ in $G'$ contain $e$, they must have the same $z$. Now the degree of $z$ in $H_2$ is at least
$|G'| > D_0$, which contradicts the fact that $H_0$ has maximum degree at most $D_0$.

{\bf Case 2.} $b<1/\epsilon$. In this case, $d_G < 2n_0^{1/k}/\epsilon$ and so $d_0< (216k^2/\epsilon) n_0^{1/k}$. By Lemma~\ref{alphabound} on $H_0$,
\begin{equation*}
    \alpha(H) \geq \alpha(H_0) \geq \frac{2n_0}{3d_0^{\frac{1}{2}}} \geq \frac{2}{3}\l(\frac{216k^2}{\epsilon}\r)^{-\frac{1}{2}} n_0^{\frac{2k-1}{2k}}\ge\frac{1}{9\sqrt{6}k}n^{\frac{2k-1}{2k}-\frac{5k-2}{2k\sqrt{\log_2 n}}}>n^{\frac{2k-1}{2k}-\frac{5}{2\sqrt{\log n}}}.
\end{equation*}
Now let $n=t^{\frac{2k}{2k-1}+\frac{4}{\sqrt{\log t}}}$. Clearly, $\log n> \frac{2k}{2k-1}\log t$. Hence, an $n$-vertex $\mathcal{B}_{2k}$-free 3-graph $H$ contains an independent set of size
$$
n^{\frac{2k-1}{2k}-\frac{5}{2\sqrt{\log n}}}=t^{(\frac{2k}{2k-1}+\frac{4}{\sqrt{\log t}})(\frac{2k-1}{2k}-\frac{5}{2\sqrt{\log n}})}>t
$$
provided $n$ is large enough. Therefore, we have $R(t,\mathcal{B}_{2k})<t^{\frac{2k}{2k-1}+\frac{4}{\sqrt{\log t}}}$.
\end{proof}
In fact, by more careful computation, we can obtain a slightly better upper bound $R(t,\mathcal{B}_{2k})<t^{\frac{2k}{2k-1}+\frac{c}{\sqrt{\log t}}}$, where $c>\frac{5k-2}{2k-1}\cdot\sqrt{\frac{(2k)\log 2}{2k-1}}$.

\section{Concluding remarks}
\begin{itemize}
    \item Notice that Theorem 2 is valid for odd values of $k$, we believe that Theorem 1 should extend to odd values of $k$. An obstacle to applying the same idea as in the proof for even values of $k$ is that we don't have ``good'' supersaturation for odd cycles. New ideas may be required to complete the proof for odd values.
    \item It seems likely that Theorem 1 can be extended to $r$-uniform hypergraphs with $r \geq 4$, however when following the proof of Theorem 1, two obstacles arise. The first is that one requires supersaturation for Berge cycles in $r$-uniform hypergraphs for $r \geq 3$ (in other words, an $r$-uniform version of Lemma 8). A second obstacle is that an $r$-uniform analog of Lemma 9 is not straightforward: for instance if an edge $e$ in an $r$-graph is contained in $m$ Berge cycles of length $2k$, then the number of edges may be as low as $m^{1/(2k - 1)}$: take a graph $2k$-cycle, and replace one edge with the hyperedge $e$, and each other edge with $m^{1/(2k - 1)}$ hyperedges. We believe these technical obstacles may be overcome (some of the ideas in the recent paper of Mubayi and Yepremyan~\cite{MY} may apply).
\end{itemize}

\section{Acknowledgments}
We would like to thank the anonymous referees for their careful reading of the paper and helpful suggestions. In particular, one of the referee's comments on Berge cycles without non-triviality leads to Theorem~\ref{thm:4bc}.

\end{document}